\documentclass[11pt]{article}
       \usepackage{amsfonts}
       \usepackage{stmaryrd}
       \usepackage{latexsym,amssymb,mathrsfs,fancyhdr}
       \font\tenmsb=msbm10
       \font\sevenmsb=msbm7
       \font\fivemsb=msbm5
       \catcode`\@=11
       \ifx\amstexloaded@\relax\catcode`\@=\active
       \endinput\else\let\amstexloaded@\relax\fi
       \def\spaces@{\space\space\space\space\space}
       \def\spaces@@{\spaces@\spaces@\spaces@\spaces@\spaces@}
       \def\space@.  {\futurelet\space@\relax}
       \space@.   %
       \def\Err@#1{\errhelp\defaulthelp@\errmessage{AmS-TeX error: #1}}
       \def\relaxnext@{\let\next\relax}
       \def\accentfam@{7}
       \def\noaccents@{\def\accentfam@{0}}
       \def\Cal{\relaxnext@\ifmmode\let\next\Cal@\else
       \def\next{\Err@{Use \string\Cal\space only in math mode}}\fi\next}
       \def\Cal@#1{{\Cal@@{#1}}}
       \def\Cal@@#1{\noaccents@\fam\tw@#1}
       \def\Bbb{\relaxnext@\ifmmode\let\next\Bbb@\else
       \def\next{\Err@{Use \string\Bbb\space only in math mode}}\fi\next}
       \def\Bbb@#1{{\Bbb@@{#1}}}
       \def\Bbb@@#1{\noaccents@\fam\msbfam#1}
       \newfam\msbfam
       \textfont\msbfam=\tenmsb
       \scriptfont\msbfam=\sevenmsb
       \scriptscriptfont\msbfam=\fivemsb
\usepackage{dsfont}
\usepackage[german,english]{babel}
\usepackage{amsmath,amssymb}
\usepackage[square, comma, sort&compress, numbers]{natbib}
\usepackage{cases}
\usepackage{mathrsfs}
\usepackage{amsmath}
\usepackage{amsthm}
\usepackage{amsfonts}
\usepackage{amssymb}
\usepackage{latexsym}
\usepackage{fancyhdr}
\usepackage{extarrows}

\newtheorem{thm}{Theorem}[section]
\newtheorem{prop}[thm]{Proposition}
\newtheorem{lem}[thm]{Lemma}
\newtheorem{rem}[thm]{Remark}
\newtheorem{iteration lemma}[thm]{iteration Lemma}
\newtheorem{cor}[thm]{Corollary}
\newtheorem{defn}[thm]{Definition}

\newtheorem{eg}[thm]{Example}

\newtheorem*{acknowledgements*}{ACKNOWLEDGEMENTS}

\begin{document}

\setlength{\columnsep}{5pt}
\title{\bf New characterizations for core inverses in rings with involution}
\author{Sanzhang  Xu\footnote{ E-mail: xusanzhang5222@126.com},
\ Jianlong Chen\footnote{ Corresponding author. E-mail: jlchen@seu.edu.cn},
\ Xiaoxiang  Zhang\footnote{  E-mail: z990303@seu.edu.cn } \\
Department of  Mathematics, Southeast University \\  Nanjing 210096,  China }
     \date{}

\maketitle
\begin{quote}
{\textbf{}\small
 The core inverse for a complex matrix was introduced by Baksalary and Trenkler.
 Raki\'{c}, Din\v{c}i\'{c} and Djordjevi\'{c} generalized the core inverse of a complex matrix to the case of an element in a ring.
 They also proved that the core inverse of an element in a ring can be characterized by five equations
 and every core invertible element is group invertible.
 It is natural to ask when a group invertible element is core invertible,
 in this paper, we will answer this question. We will use three equations
 to characterize the core inverse of an element. That is, let $a, b\in R$,
 then $a\in R^{\tiny\textcircled{\tiny\#}}$ with $a^{\tiny\textcircled{\tiny\#}}=b$ if and only if
 $(ab)^{\ast}=ab$, $ba^{2}=a$ and $ab^{2}=b$.
Finally, we investigate the additive property of two core invertible elements. Moreover,
the formulae of the sum of two core invertible elements are presented.


}
\end{quote}

\section{ Introduction }\label{a}
The core inverse and the dual core inverse for a complex matrix were introduced by Baksalary and Trenkler \cite{BT}.
Let $A\in M_{n}(\mathbb{C})$, where $M_{n}(\mathbb{C})$ denotes the ring of all $n\times n$ complex matrices. A matrix $X\in M_{n}(\mathbb{C})$ is called a core inverse of $A$, if it satisfies
$AX=P_{A}$ and $\mathcal{R}(X)\subseteq \mathcal{R}(A)$,
where $\mathcal{R}(A)$ denotes the column space of $A$,
and $P_{A}$ is the orthogonal projector onto $\mathcal{R}(A)$.
And if such a matrix exists, then it is unique (and denoted by $A^{\tiny\textcircled{\tiny\#}}$).
Baksalary and Trenkler gave several
characterizations of the core inverse by using the decomposition of Hartwig and Spindelb\"{o}ck.

In \cite{RDD1} Raki\'{c}, Din\v{c}i\'{c} and Djordjevi\'{c} generalized the core inverse of a complex matrix
to the case of an operator in $\mathcal{L}(H)$, where $H$ is a Hilbert space and $\mathcal{L}(H)$ denotes the set of all
bounded linear operators from $H$ to $H$. They also proved that the core inverse of an operator $A\in\mathcal{L}(H)$ can be
determined by five operator equations.

In \cite{RDD} Raki\'{c}, Din\v{c}i\'{c} and Djordjevi\'{c} generalized the core inverse of a complex matrix to the case of an element in a ring.
They proved that a core invertible element is group invertible and for $a\in R$, the core inverse of $a$ is the unique element $x$ satisfying the following five equations:
$$axa=a,~~xax=x,~~(ax)^{\ast}=ax,~~xa^{2}=a,~~ax^{2}=x.$$

In \cite[Theorem 2.1]{WL}, Wang and Liu proved that if $A\in \mathbb{C}^{CM}_{n}$, where $\mathbb{C}^{CM}_{n}=\{A\in M_{n}(\mathbb{C})\mid r(A)=r(A^{2})\}$, then the core inverse of $A$
is the unique matrix $X\in M_{n}(\mathbb{C})$ satisfying the following three equations:
$$~AXA=A,~~~~AX^{2}=X,~~~~(AX)^{\ast}=AX.$$

Motivated by \cite{RDD} and \cite{WL}, we answer the question when a group invertible element is core invertible
and prove that the core inverse of an element in a ring can be characterized by three equations.
We also extend the result \cite[Theorem 2.1]{WL} to the general ring. Let $a,b\in R$,
in Theorem \ref{coreinverse2}, we show that $a\in R^{\tiny{\textcircled{\tiny\#}}}$
if and only if $a\in R^\#\cap R^{\{1,3\}}$.
In Theorem \ref{coreinverse1}, we show that $a\in R^{\tiny\textcircled{\tiny\#}}$ with $a^{\tiny\textcircled{\tiny\#}}=b$
if and only if $(ab)^{\ast}=ab$, $ba^{2}=a$ and $ab^{2}=b$.
In Theorem \ref{Thr-conditions-cor1}, we show that
if $Ra=Ra^{2}$, then $a\in R^{\tiny\textcircled{\tiny\#}}$ with core inverse $b$
if and only if $aba=a$, $(ab)^{\ast}=ab$ and $ab^{2}=b$.

The problem of the Moore-Penrose inverse of the sum of two Moore-Penrose invertible elements in complex matrix ring was first considered by Penrose in \cite{P}.
Many scholars focus on the additive problem of two generalized invertible elements,
such as \cite{BXT,CG,CD,DW,PH,ZCDW}. The problem of the Drazin inverse of the sum of two Drazin invertible elements in a
ring was first considered by Drazin in \cite{D}. He proved that for
two group invertible elements $a,b\in R$ satisfying $ab=0=ba$, then $a+b$ is group invertible.
Hartwig, Wang and Wei in \cite{HWW} gave a generalization of above result in the complex matrix case in terms of $ab=0$.
Chen, Zhuang and Wei in \cite{CZW} also gave a generalization of above result in the morphism case in terms of $ab=0$.
In section 4, we will show that if $a$ and $b$ are core invertible and satisfy $ab=0$ and $a^{\ast}b=0$, then $a+b$ is core invertible and $(a+b)^{\tiny{\textcircled{\tiny\#}}}=b^{\pi}a^{\tiny{\textcircled{\tiny\#}}}+b^{\tiny{\textcircled{\tiny\#}}}$, where $b^{\pi}=1-b^{\tiny{\textcircled{\tiny\#}}}b$.

For the convenience of the reader, some fundamental concepts are given as follows.
An element $b\in R$ is an inner inverse of $a\in R$ if $aba=a$ holds. The set of all inner inverses of $a$ will be denoted by $a\{1\}$.
An element $a\in R$ is said to be group invertible
if there exists $b\in R$ such that the following equations hold:
$$aba=a, \quad bab=b, \quad ab=ba.$$
The element $b$ which satisfies the above equations is called a group inverse of $a$.
If such an element $b$ exists, it is unique and denoted by $a^\#$. The set of all group invertible elements will be denoted by $R^\#$.

Let $R$ be a $*$-ring, that is a ring with an involution $a\mapsto a^*$
satisfying $(a^*)^*=a$, $(ab)^*=b^*a^*$ and $(a+b)^*=a^*+b^*$.

An element $\tilde{a}\in R$  is called a $\{1,3\}$-inverse of $a$ if $a\tilde{a}a=a$ and $(a\tilde{a})^{\ast}=a\tilde{a}$. The set of all $\{1,3\}$-invertible elements will be denoted by $R^{\{1,3\}}$. Similarly, an element $\hat{a}\in R$  is called a $\{1,4\}$-inverse of $a$ if $a\hat{a}a=a$ and $(\hat{a}a)^{\ast}=\hat{a}a$. The set of all $\{1,4\}$-invertible elements will be denoted by $R^{\{1,4\}}$.

We will also use the following notations: $aR=\{ax\mid x\in R\}$, $Ra=\{xa\mid x\in R\}$, $^{\circ}a=\{x\in R\mid xa=0\}$, $a^{\circ}=\{x\in R\mid ax=0\}$.

\section{ When a group invertible element is core invertible }\label{a}

In this section, some characterizations of the existence of a core invertible element in rings are obtained. Let us start this section with some preliminaries.

\begin{defn} \cite{RDD}
Let $a,x\in R$, if
$$axa=a,~xR=aR,~Rx=Ra^{\ast},$$
then $x$ is called a core inverse of $a$ and if such an element $x$ exists, then it is unique and denoted by $a^{\tiny{\textcircled{\tiny\#}}}$. The set of all core invertible elements in $R$ will be denoted by $R^{\tiny{\textcircled{\tiny\#}}}$.
\end{defn}

\begin{defn} \cite{RDD}
Let $a,x\in R$, if
$$axa=a,~xR=a^{\ast}R,~Rx=Ra,$$
then $x$ is called a dual core inverse of $a$ and if such an element $x$ exists, then it is unique and denoted by $a_{\tiny{\textcircled{\tiny\#}}}$. The set of all dual core invertible elements in $R$ will be denoted by $R_{\tiny{\textcircled{\tiny\#}}}$.
\end{defn}

\begin{lem} \cite[Theorem 2.14]{RDD} \label{five-equations-yy}
Let $a\in R$, then $a\in R^{\tiny\textcircled{\tiny\#}}$ with core inverse $x$ if and only if
$$axa=a,~~xax=x,~~(ax)^{\ast}=ax,~~xa^{2}=a,~~ax^{2}=x.$$
\end{lem}

\begin{lem} \cite{RDD} \label{five-equations}
Let $a\in R$, we have:\\
$(1)$ If $a\in R^{\tiny{\textcircled{\tiny\#}}}$, then $a\in R^\#$ and $a^\#=(a^{\tiny{\textcircled{\tiny\#}}})^{2}a$;\\
$(2)$ $a\in R^\#$ if and only if there exists an idempotent $q\in R$ such that $qR=aR$ and $Rq=Ra.$
\end{lem}

\begin{lem} \cite[Theorem 2]{HC} \label{chen}
Let $a\in R$, then $a\in R^{\{1,3\}}$ if and only if there exists unique $p\in R$ such that $p^{2}=p=p^{\ast}$ and $aR=pR$.
\end{lem}

In \cite[Theorem 2.14]{RDD}, Raki\'{c}, Din\v{c}i\'{c} and Djordjevi\'{c} proved that $a\in R^{\tiny\textcircled{\tiny\#}}$ if and only if there exists  $p^{2}=p=p^{\ast}$ and $q^{2}=q$ such that $aR=pR$ and $qR=aR$ and $Rq=Ra.$ Thus by Lemma \ref{five-equations} and Lemma \ref{chen}, we have the following theorem. For the convenience of the reader, we give another method to prove this result.

\begin{thm} \label{coreinverse2}
Let $a\in R$, then $a\in R^{\tiny{\textcircled{\tiny\#}}}$ if and only if $a\in R^\#\cap R^{\{1,3\}}$.
In this case, $a^{\tiny{\textcircled{\tiny\#}}}=a^\#aa^{(1,3)}.$
\end{thm}
\begin{proof} Suppose $a\in R^{\tiny{\textcircled{\tiny\#}}}$, then $a^\#=(a^{\tiny{\textcircled{\tiny\#}}})^{2}a$ by  Lemma \ref{five-equations} and $a\in R^{\{1,3\}}$ by Lemma \ref{five-equations-yy}. Conversely, suppose $a\in R^\#\cap R^{\{1,3\}}$, then $aa^{(1,3)}a=a$ and $(aa^{(1,3)})^{\ast}=aa^{(1,3)}$, we have
$$a=aa^{(1,3)}a=(aa^{(1,3)})^{\ast}a=(a^{(1,3)})^{\ast}a^{\ast}a.$$
Let $y=a^\#aa^{(1,3)}$, then
$aya=aa^\#aa^{(1,3)}a=aa^\#a=a$.
Since $y=a^\#aa^{(1,3)}=a(a^\#)^{2}aa^{(1,3)}$
and $a=a^\#a^{2}=a^\#(a^{(1,3)})^{\ast}a^{\ast}a^{2}=ya^{2}$, we get
$yR=aR.$
We also have $$a^{\ast}=a^{\ast}aa^{(1,3)}=a^{\ast}aa^\#aa^{(1,3)}=a^{\ast}ay.$$ So
$Ry=Ra^{\ast}.$ Thus $a\in R^{\tiny{\textcircled{\tiny\#}}}$ by the definition of the core inverse.
\end{proof}

\begin{cor}
Let $a\in R$, then the following conditions are equivalent:\\
$(1)$ $a\in R^{\tiny{\textcircled{\tiny\#}}}$;\\
$(2)$ $a\in R^\#$ and there exists $x\in R$ such that $(ax)^{\ast}=ax$ and $xa^{2}=a$;\\
$(3)$ $a\in R^\#$ and there exists $x\in R$ such that $(ax)^{\ast}=ax$ and $xa=a^\#a$.
\end{cor}
\begin{proof}
$(1)\Rightarrow(2)$ It is clear by Lemma \ref{five-equations-yy} and Lemma \ref{five-equations}.

$(2)\Rightarrow(3)$  It is sufficient to prove that $xa^{2}=a$ implies $xa=a^\#a$. It is easy to see that by $xa=xa^{2}a^\#=aa^\#=a^\#a.$

$(3)\Rightarrow(1)$  Suppose $a\in R^\#$ and there exists $x\in R$ such that $(ax)^{\ast}=ax$ and $xa=a^\#a$.
Then $axa=aa^\#a=a$, that is $a\in R^{\{1,3\}}$. Thus by the hypothesis $a\in R^\#$ and Theorem \ref{coreinverse2},
we have $a\in R^{\tiny{\textcircled{\tiny\#}}}$.
\end{proof}

There exists a corresponding result for the dual core inverse.

\begin{thm} \label{dualcoreinverse2}
Let $a\in R$, then $a\in R_{\tiny{\textcircled{\tiny\#}}}$ if and only if $a\in R^\#\cap R^{\{1,4\}}$.
In this case, $a_{\tiny{\textcircled{\tiny\#}}}=a^{(1,4)}aa^\#.$
\end{thm}

\begin{lem}  \cite{HC} \label{13-14-equations}
Let $a\in R$, we have:\\
$(I)$ The following conditions are equivalent:\\
$(1)$ $a\in R^{\{1,3\}}$;\\
$(2)$ $R=aR\!\oplus\!(a^{\ast})^{\circ}$;\\
$(3)$ $R=aR\!\!+\!\!(a^{\ast})^{\circ}$;\\
$(4)$ $R=Ra^{\ast}\!\!\oplus^{\circ}\!\!a$;\\
$(5)$ $R=Ra^{\ast}\!\!+ ^{\circ}\!\!a$.\\
In this case,
\begin{eqnarray}\label{eq4}
a\{1,3\}=\{r+(1-ra)w_{1}~|~w_{1}\in R\}=\{s^{\ast}+(1-s^{\ast}a)w_{2}~|~w_{2}\in R\},
\end{eqnarray}
where $1=ar+u=sa^{\ast}+v$ for some $r,s\in R$, $u\in (a^{\ast})^{\circ}$ and $v\in ^{\circ}\!\!a$.\\
$(II)$ The following conditions are equivalent:\\
$(1)$ $a\in R^{\{1,4\}}$;\\
$(2)$ $R=a^{\ast}\!R\!\oplus \!a^{\circ}$;\\
$(3)$ $R=a^{\ast}\!R\!+\!a^{\circ}$;\\
$(4)$ $R=Ra\!\oplus\!^{\circ}\!(a^{\ast})$;\\
$(5)$ $R=Ra\!+ \!^{\circ}\!(a^{\ast})$.\\
In this case,
\begin{eqnarray}\label{eq5}
a\{1,4\}=\{s+w_{1}(1-as)~|~w_{1}\in R\}=\{r^{\ast}+w_{2}(1-ar^{\ast})~|~w_{2}\in R\},
\end{eqnarray}
where $1=a^{\ast}r+u=sa+v$ for some $r,s\in R$, $u\in a^{\circ}$ and $v\in ^{\circ}\!\!(a^{\ast})$.
\end{lem}

\begin{lem} \cite{H} \label{group-decom}
Let $a\in R$, then $a\in R^\#$ if and only if
$R=aR\oplus a^{\circ}$ if and only if
$R=Ra\oplus ^{\circ}\!\!a.$ \\
In this case,
\begin{eqnarray}\label{eq6}
a^\#=ax^{2}=y^{2}a,
\end{eqnarray}
where $1=ax+u=ya+v$ for some $x,y\in R$, $u\in a^{\circ}$ and $v\in ^{\circ}\!\!a$.
\end{lem}

\begin{prop} \label{R-decomp-core-inverse}
Let $a\in R$, then the following conditions are equivalent:\\
$(1)$ $a\in R^{\tiny{\textcircled{\tiny\#}}}$;\\
$(2)$ $R=aR\oplus(a^{\ast})^{\circ}=aR\oplus a^{\circ}$;\\
$(3)$ $R=aR+(a^{\ast})^{\circ}=aR\oplus a^{\circ}$;\\
$(4)$ $R=Ra^{\ast}\oplus^{\circ}\!\!a=aR\oplus a^{\circ}$;\\
$(5)$ $R=Ra^{\ast}+ ^{\circ}\!\!a=aR\oplus a^{\circ}$;\\
$(6)$ $R=aR\oplus(a^{\ast})^{\circ}=Ra\oplus ^{\circ}\!\!a$;\\
$(7)$ $R=aR+(a^{\ast})^{\circ}=Ra\oplus ^{\circ}\!\!a$;\\
$(8)$ $R=Ra^{\ast}\oplus^{\circ}\!\!a=Ra\oplus ^{\circ}\!\!a$;\\
$(9)$ $R=Ra^{\ast}+ ^{\circ}\!\!a=Ra\oplus ^{\circ}\!\!a$.\\
In this case, $$a^{\tiny{\textcircled{\tiny\#}}}=ay_{1}^{2}ax_{1}=ay_{1}^{2}ax_{2}^{\ast}=y_{2}ax_{1}=y_{2}ax_{2}^{\ast},$$
where $1=ax_{1}+u_{1}=x_{2}a^{\ast}+u_{2}=ay_{1}+v_{1}=y_{2}a+v_{2}$ for some $x_{1},x_{2},y_{1},y_{2}\in R$,
 $u_{1}\in (a^{\ast})^{\circ}$, $v_{1}\in a^{\circ}$ and $u_{2},v_{2}\in ^{\circ}\!\!a$.
\end{prop}
\begin{proof} By Theorem \ref{coreinverse2}, we have $a\in R^{\tiny{\textcircled{\tiny\#}}}$ if and only if $a\in R^\#\cap R^{\{1,3\}}$ and
\begin{eqnarray}\label{eq7}
a^{\tiny{\textcircled{\tiny\#}}}=a^\#aa^{(1,3)}.
\end{eqnarray}
Thus it is easy to see $(1)$-$(9)$ are equivalent by Lemma \ref{13-14-equations} and Lemma \ref{group-decom}.
Suppose $1=ax_{1}+u_{1}=x_{2}a^{\ast}+u_{2}=ay_{1}+v_{1}=y_{2}a+v_{2}$, for some $x_{1},x_{2},y_{1},y_{2}\in R$,
 $u_{1}\in (a^{\ast})^{\circ}$, $v_{1}\in a^{\circ}$ and $u_{2},v_{2}\in ^{\circ}\!\!a$. Then by (\ref{eq4}) and (\ref{eq6}), respectively, we have
\begin{eqnarray}\label{eq8}
a^{(1,3)}=x_{1}=x_{2}^{\ast}~~~and~~~a^\#=ay_{1}^{2}=y_{2}^{2}a.
\end{eqnarray}
\begin{eqnarray}\label{eq9}
a=y_{2}a^{2}+v_{2}a=y_{2}a^{2}.
\end{eqnarray}
Thus by (\ref{eq7}) and (\ref{eq8}), we have
$$a^{\tiny{\textcircled{\tiny\#}}}=ay_{1}^{2}ax_{1}=ay_{1}^{2}ax_{2}^{\ast}=y_{2}^{2}a^{2}x_{1}=y_{2}^{2}a^{2}x_{2}^{\ast}.$$
Hence $a^{\tiny{\textcircled{\tiny\#}}}=ay_{1}^{2}ax_{1}=ay_{1}^{2}ax_{2}^{\ast}=y_{2}ax_{1}=y_{2}ax_{2}^{\ast}$ by (\ref{eq9}).
\end{proof}

\begin{prop} \label{R-decomp-dual-core-inverse}
Let $a\in R$, then the following conditions are equivalent:\\
$(1)$ $a\in R_{\tiny{\textcircled{\tiny\#}}}$;\\
$(2)$ $R=a^{\ast}R\oplus a^{\circ}=aR\oplus a^{\circ}$;\\
$(3)$ $R=a^{\ast}R+ a^{\circ}=aR\oplus a^{\circ}$;\\
$(4)$ $R=Ra\oplus^{\circ}\!(a^{\ast})=aR\oplus a^{\circ}$;\\
$(5)$ $R=Ra+^{\circ}\!(a^{\ast})=aR\oplus a^{\circ}$;\\
$(6)$ $R=a^{\ast}R\oplus a^{\circ}=Ra\oplus ^{\circ}\!\!a$;\\
$(7)$ $R=a^{\ast}R+ a^{\circ}=Ra\oplus ^{\circ}\!\!a$;\\
$(8)$ $R=Ra\oplus^{\circ}\!(a^{\ast})=Ra\oplus ^{\circ}\!\!a$;\\
$(9)$ $R=Ra+^{\circ}\!(a^{\ast})=Ra\oplus ^{\circ}\!\!a$.\\
In this case, $$a_{\tiny{\textcircled{\tiny\#}}}=x_{1}^{\ast}ay_{1}=x_{1}^{\ast}ay_{2}^{2}a=x_{2}ay_{1}=x_{2}ay_{2}^{2}a,$$
where $1=a^{\ast}x_{1}+u_{1}=x_{2}a+u_{2}=ay_{1}+v_{1}=y_{2}a+v_{2}$ for some $x_{1},x_{2},y_{1},y_{2}\in R$,
 $u_{2}\in ^{\circ}\!(a^{\ast})$, $v_{2}\in ^{\circ}\!\!a$ and $u_{1},v_{1}\in a^{\circ}$.
\end{prop}

\section{ New characterizations of elements to be core invertible by equations  }\label{a}

In this section, some new characterizations of core and dual core inverses in rings are obtained.

In the following theorem, we can prove that
equations $axa=a$ and $xax=x$ in Lemma \ref{five-equations-yy} can be dropped.

\begin{thm} \label{coreinverse1}
Let $a,x\in R$, then $a\in R^{\tiny{\textcircled{\tiny\#}}}$ with $a^{\tiny{\textcircled{\tiny\#}}}=x$
if and only if $(ax)^{\ast}=ax$, $xa^{2}=a$ and $ax^{2}=x$.
\end{thm}
\begin{proof} Suppose $a\in R^{\tiny{\textcircled{\tiny\#}}}$, then we have $(ax)^{\ast}=ax$, $xa^{2}=a$ and $ax^{2}=x$
 by Lemma \ref{five-equations-yy}.
Conversely, if $(ax)^{\ast}=ax$, $xa^{2}=a$ and $ax^{2}=x$, then we have
\begin{eqnarray}\label{eq1}
&x=ax^{2}=xa^{2}x^{2}=xa(ax^{2})=xax.\\
&a=xa^{2}=ax^{2}a^{2}=ax(xa^{2})=axa.
\end{eqnarray}
Thus by Lemma \ref{five-equations-yy}, we have $a\in R^{\tiny\textcircled{\tiny\#}}$ and $a^{\tiny\textcircled{\tiny\#}}=x.$
\end{proof}

In \cite[Theorem 2.1]{WL}, Wang and Liu proved that for $A\in \mathbb{C}^{CM}_{n}$, then the core inverse of $A$
is the unique matrix $X\in M_{n}(\mathbb{C})$ satisfying the following three equations:
$$~AXA=A,~~~~AX^{2}=X,~~~~(AX)^{\ast}=AX.$$
By the definition of $\mathbb{C}^{CM}_{n}$, we know $A\in \mathbb{C}^{CM}_{n}$ if and only if $A$ is group invertible.
We will extend the above result to the ring case, let us begin with an lemma.
\begin{lem} \cite{BG} \label{Gi-conditions}
Let $R$ be a ring, if $a\in R$ is regular with inner inverse $a^{-}$, then the following conditions are equivalent:\\
$(1)$ $a\in R^\#$;\\
$(2)$ $u=a^{2}a^{-}+1-aa^{-}$ is invertible;\\
$(3)$ $v=a^{-}a^{2}+1-a^{-}a$ is invertible;\\
$(4)$ $u^{\tiny\prime}=a+1-aa^{-}$ is invertible;\\
$(5)$ $v^{\tiny\prime}=a+1-a^{-}a$ is invertible.\\
In this case,
$$a^\#=u^{-2}a=av^{-2}.$$
\end{lem}

\begin{thm} \label{Thr-conditions-cor1}
Let $a,b\in R$. If $Ra=Ra^{2}$, then the following are equivalent:\\
$(1)$ $a\in R^{\tiny\textcircled{\tiny\#}}$ with core inverse $b$;\\
$(2)$ $aba=a$, $(ab)^{\ast}=ab$ and $ab^{2}=b$.
\end{thm}
\begin{proof}

$(1)\Rightarrow(2)$ It is trivial by Lemma \ref{five-equations-yy}.

$(2)\Rightarrow(1)$
Suppose $aba=a$, $ab^{2}=b$, $(ab)^{\ast}=ab$, then $a=aba=a^{2}b^{2}a\in a^{2}R$, thus
$a\in R^\#$ by the hypothesis $Ra=Ra^{2}$ and  \cite[Proposition 7]{H}. So $a+1-ab$ is invertible by Lemma \ref{Gi-conditions}.
\begin{eqnarray}\label{eq10}
(a+1-ab)a^\#ab=aa^\#ab=ab.
\end{eqnarray}
(\ref{eq10}) is equivalent to
\begin{eqnarray}\label{eq11}
a^\#ab=(a+1-ab)^{-1}ab.
\end{eqnarray}
By (\ref{eq11}) and Theorem \ref{coreinverse2} we have
\begin{eqnarray}\label{eq12}
a^{\tiny\textcircled{\tiny\#}}=(a+1-ab)^{-1}ab.
\end{eqnarray}
Since $ab^{2}=b$, then
$0=b-ab^{2}=(a+1-ab)b-ab$, thus
\begin{eqnarray}\label{eq13}
b=(a+1-ab)^{-1}ab.
\end{eqnarray}

Hence $b=a^{\tiny\textcircled{\tiny\#}}$ by (\ref{eq12}) and (\ref{eq13}).
\end{proof}

An infinite matrix $M$ is said to be bi-finite if it is both row-finite and column-finite.

\begin{eg} \label{conter1}
The condition $Ra=Ra^{2}$ in Theorem $\ref{Thr-conditions-cor1}$ cannot be dropped.
Let $R$ be the ring of all bi-finite matrices over the real field $\mathbb{R}$ with transpose as involution.
Consider the following matrices $A,B$ over $R$.
$A=\sum\limits_{i=1}^{\infty}e_{i,i+1}$ and
$B=A^{\ast}$, where $e_{i,j}$ denotes the matrix with $(i,j)$-th is $1$ and other entries are zero.
Then
$BA=\sum\limits_{i=2}^{\infty}e_{i,i}$ and
$AB=I$, thus $ABA=A$, $(AB)^{\ast}=AB$, $AB^{2}=B$. Since $A$ is not group invertible, $A$ is not core invertible.
\end{eg}

\begin{thm} \label{Thr-conditions-cor2}
Let $a,b\in R$. If $aR=a^{2}R$, then the following are equivalent:\\
$(1)$ $a\in R^{\tiny\textcircled{\tiny\#}}$ with core inverse $a^{\tiny\textcircled{\tiny\#}}=b$;\\
$(2)$ $bab=b$, $(ab)^{\ast}=ab$ and $ba^{2}=a$.
\end{thm}

\begin{proof}
$(1)\Rightarrow(2)$ It is trivial by Lemma \ref{five-equations-yy}.

$(2)\Rightarrow(1)$
Suppose $bab=b$, $(ab)^{\ast}=ab$ and $ba^{2}=a$, then $a=ba^{2}\in Ra^{2}$, thus
$a\in R^\#$ by the hypothesis $aR=a^{2}R$ and  \cite[Proposition 7]{H}.
Post multiplication by $a^\#$ on equation $ba^{2}=a$ yields $ba=aa^\#$, thus $$aba=a^{2}a^\#=a.$$
Hence $b\in a\{1,3\}$ by $(ab)^{\ast}=ab$.
By Theorem \ref{coreinverse2}, we have $a^{\tiny{\textcircled{\tiny\#}}}=a^\#aa^{(1,3)}$. Therefore
$$a^{\tiny\textcircled{\tiny\#}}=a^\#aa^{(1,3)}=aa^\#b=bab=b.$$
\end{proof}

\begin{eg} \label{conter2}
The condition $aR=a^{2}R$ in Theorem $\ref{Thr-conditions-cor2}$ cannot be dropped.
Let $R$ be the ring of all bi-finite matrices with entries over field $\mathbb{R}$ with transpose as involution.
Consider the following matrices $A,B$ over $R$.
$A=\sum\limits_{i=1}^{\infty}e_{i+1,i}$ and
$B=A^{\ast}$
, then
$AB=\sum\limits_{i=2}^{\infty}e_{i,i}$,
$BA=I$, thus $BAB=B$, $(AB)^{\ast}=AB$ and $BA^{2}=A$. Since
$A$ is not group invertible, whence $A$ is not core invertible.
\end{eg}

By Theorem \ref{Thr-conditions-cor1} and Theorem \ref{Thr-conditions-cor2}, we have the following  corollary.

\begin{cor} \label{Thr-conditions}
Let $a,b\in R$. If $a\in R^\#$, then the following conditions are equivalent:\\
$(1)$ $a\in R^{\tiny\textcircled{\tiny\#}}$ with core inverse $b$;\\
$(2)$ $aba=a$, $ab^{2}=b$, $(ab)^{\ast}=ab$;\\
$(3)$ $bab=b$, $ba^{2}=a$, $(ab)^{\ast}=ab$.
\end{cor}

\begin{lem} \cite{PH1} \label{a=aa-aa=a}
Let $a\in R$, if $a$ is regular and $a^{-}\in a\{1\}$, then we have:\\
$(1)$ $aR=a^{2}R$ if and only if $U_{1}=a+1-aa^{-}$ is right invertible if and only if $V_{1}=a+1-a^{-}a$ is right invertible;\\
$(2)$ $Ra=Ra^{2}$ if and only if $U_{2}=a+1-aa^{-}$ is left invertible if and only if $V_{2}=a+1-a^{-}a$ is left invertible.
\end{lem}

A ring $R$ is called direct finite, if $ab=1$ implies $ba=1$ for any $a,b\in R.$

\begin{thm} \label{direct-finte}
The following conditions are equivalent:\\
$(1)$ $R$ is a direct finite ring;\\
$(2)$ Let $a,b\in R$, then $aba=a$, $(ab)^{\ast}=ab$, $ab^{2}=b$ if and only if $a\in R^{\tiny\textcircled{\tiny\#}}$ with $a^{\tiny\textcircled{\tiny\#}}=b$.
\end{thm}
\begin{proof}
$(1)\Rightarrow(2)$ Suppose $a\in R^{\tiny\textcircled{\tiny\#}}$ with $a^{\tiny\textcircled{\tiny\#}}=b$,
then $aba=a$, $(ab)^{\ast}=ab$, $ab^{2}=b$ by Lemma \ref{five-equations-yy}.
Conversely, suppose $R$ is a direct finite ring and
$aba=a$, $(ab)^{\ast}=ab$, $ab^{2}=b$, then $b\in a\{1,3\}$.
Since $$a=aba=a(ab^{2})a=a^{2}b^{2}a\in a^{2}R,$$ so
$a+1-ab$ is right invertible by Lemma \ref{a=aa-aa=a}. Thus
$a+1-ab$ is invertible by $R$ is a direct finite ring. By Lemma \ref{Gi-conditions}, we have $a\in R^\#$, therefore $a^{\tiny\textcircled{\tiny\#}}=b$
by Corollary \ref{Thr-conditions}.

$(2)\Rightarrow(1)$ For arbitrary $m,n\in R$, if $mn=1$, then $mnm=m$, $(mn)^{\ast}=mn$, $mn^{2}=n$, thus by the hypothetical condition, we have
$m\in R^{\tiny\textcircled{\tiny\#}}$ with $m^{\tiny\textcircled{\tiny\#}}=n$, then $m\in R^\#$ by Lemma \ref{five-equations}.
Post-multiplication by $n$ on equation $m=mm^\#m$ yields
$$1=mn=mm^\#mn=mm^\#,$$
thus, $mm^\#=m^\#m=1$ which imply $m$ is invertible. Then $nm=1$, so $R$ is a direct finite ring.
\end{proof}

\begin{thm}
Consider the following  conditions:\\
$(1)$ $R$ is a direct finite ring;\\
$(2)$ Let $a,b\in R$, if $bab=b$, $(ab)^{\ast}=ab$ and $ba^{2}=a$, then $a\in R^{\tiny\textcircled{\tiny\#}}$ with $a^{\tiny\textcircled{\tiny\#}}=b$;\\
$(3)$ Let $a,b\in R$, if $aba=a$, $(ab)^{\ast}=ab$ and $ba^{2}=a$, then $a\in R^{\tiny\textcircled{\tiny\#}}$ with $a^{\tiny\textcircled{\tiny\#}}=bab$;\\
$(4)$ Let $a\in R$, if $a^{\ast}a=1$, then $aa^{\ast}=1.$

Then, we have $(1)\Rightarrow(2)\Rightarrow(3)\Rightarrow(4).$
\end{thm}
\begin{proof}
$(1)\Rightarrow(2)$
Suppose $R$ is a direct finite ring, by $bab=b$ and $ba^{2}=a$, we have
$$b=bab=b(ba^{2})b=b^{2}a^{2}b\in b^{2}R.$$
Since $b=bab$ gives $a\in b\{1\}$, by Lemma \ref{a=aa-aa=a},
$b+1-ba$ is right invertible. As $R$ is a direct finite ring, hence
$b+1-ba$ is invertible, which gives $b\in R^\#$ by Lemma \ref{Gi-conditions}.
Pre-multiplication of $b=b^{2}a^{2}b$ by $b^\#$ now yields $b^\#b=ab$, thus
$$ab^{2}=b^\#b^{2}=b.$$
Whence by Theorem \ref{coreinverse1}, we have $a\in R^{\tiny\textcircled{\tiny\#}}$ with $a^{\tiny\textcircled{\tiny\#}}=b$.

$(2)\Rightarrow(3)$
Let $c=bab$, by $bab=b$, $(ab)^{\ast}=ab$ and $ba^{2}=a$, we have
$cac=c$, $(ac)^{\ast}=ac$ and $ca^{2}=a$, thus by $(2)$, we have
$a\in R^{\tiny\textcircled{\tiny\#}}$ with $a^{\tiny\textcircled{\tiny\#}}=c=bab$.

$(3)\Rightarrow(4)$
Suppose $a^{\ast}a=1$, let $b=a^{\ast}$, then $aba=a$, $(ab)^{\ast}=ab$ and $ba^{2}=a$.
By $(3)$, we have $a\in R^{\tiny\textcircled{\tiny\#}}$, thus
$1=a^{\ast}a=a^{\ast}aa^{\tiny\textcircled{\tiny\#}}a=a^{\tiny\textcircled{\tiny\#}}a=
a(a^{\tiny\textcircled{\tiny\#}})^{2}a$, That is $a$ is invertible, which gives $aa^{\ast}=1$.
\end{proof}

\begin{prop} \label{xR-littile-aR}
Let $a,b\in R$. If $aR=a^{2}R$, then the following conditions are equivalent:\\
$(1)$ $a\in R^{\tiny\textcircled{\tiny\#}}$ with $a^{\tiny\textcircled{\tiny\#}}=b$;\\
$(2)$ $ba^{2}=a$, $(ab)^{\ast}=ab$, $bR\subseteq aR$;\\
$(3)$ $a\in R^{\{1,3\}}$ and satisfies $ba^{2}=a$, $b=ba\hat{a}$ for some $\hat{a}\in a\{1,3\}$.
\end{prop}
\begin{proof}
$(1)\Rightarrow(2)$
Suppose $a\in R^{\tiny\textcircled{\tiny\#}}$ with $a^{\tiny\textcircled{\tiny\#}}=b$, then
$ba^{2}=a$, $(ab)^{\ast}=ab$ and $b=ab^{2}$ by Lemma \ref{five-equations-yy}, thus $bR\subseteq aR.$

$(1)\Rightarrow(3)$  Suppose $a\in R^{\tiny\textcircled{\tiny\#}}$ with $a^{\tiny\textcircled{\tiny\#}}=b$, then
$ba^{2}=a$ by Lemma \ref{five-equations-yy}.
By Theorem \ref{coreinverse2}, we have
$$a^{\tiny{\textcircled{\tiny\#}}}=a^\#aa^{(1,3)}=(a^\#aa^{(1,3)})aa^{(1,3)}=a^{\tiny{\textcircled{\tiny\#}}}aa^{(1,3)}=baa^{(1,3)}.$$

$(2)\Rightarrow(1)$  Suppose $ba^{2}=a$, $(ab)^{\ast}=ab$, $bR\subseteq aR$, then $b=ax$, for some $x\in R$.
By $ba^{2}=a$, we have $b=ax=ba^{2}x=bab.$
Therefore, $a\in R^{\tiny\textcircled{\tiny\#}}$ with $a^{\tiny\textcircled{\tiny\#}}=b$ by Theorem \ref{Thr-conditions-cor2}.

$(3)\Rightarrow(1)$ Suppose $ba^{2}=a$, $b=ba\hat{a}$, for some $\hat{a}\in a\{1,3\}$ and $aR=a^{2}R$, thus $a\in R^\#$ by $a\in a^{2}R\cap Ra^{2}$. Hence
$$b=ba\hat{a}=ba^{2}a^\#\hat{a}=aa^\#\hat{a}=a^{\tiny{\textcircled{\tiny\#}}}$$
by Theorem \ref{coreinverse2}.
\end{proof}

\section{ The core invertibility of the sum of two core invertible elements }\label{a}

In this section, we will show that if two core invertible elements $a,b\in R$ satisfy $ab=0$ and $a^{\ast}b=0$, then $a+b$ is core invertible. Moreover, the explicit expression of the sum of two core invertible elements is presented.

In 1958 Drazin \cite{D} proved that if
two group invertible elements $a,b\in R$ satisfy $ab=0=ba$, then $a+b$ is group invertible.
Chen, Zhuang and Wei in \cite{CZW} gave a generalization of this result in the case of $ab=0$.
This generalization also can be found in \cite[Theorem 2.1]{BXT}.

\begin{lem} \cite{CZW} \label{additive-group}
Let $a,~b\in R$ be group invertible with $ab=0$, then $a+b$ is group invertible.
In this case,$$(a+b)^\#=(1-bb^\#)a^\#+b^\#(1-aa^\#).$$
\end{lem}

\begin{eg} \label{example1}
In $\mathbb{Z}_{8}$ with $x^{\ast}=x$ for all $x\in \mathbb{Z}_{8}$, both $a=\bar{1}$ and $b=\bar{3}$ are core invertible, but $a+b\notin R^{\tiny{\textcircled{\tiny\#}}}$.
\end{eg}

\begin{thm} \label{core-additive1}
Let $a,b\in R^{\tiny{\textcircled{\tiny\#}}}$ with core inverses $a^{\tiny{\textcircled{\tiny\#}}}$ and $b^{\tiny{\textcircled{\tiny\#}}}$, respectively. If $ab=0$ and $a^{\ast}b=0$, then $a+b\in R^{\tiny{\textcircled{\tiny\#}}}$.\\
Moreover,
$$(a+b)^{\tiny{\textcircled{\tiny\#}}}=b^{\pi}a^{\tiny{\textcircled{\tiny\#}}}+b^{\tiny{\textcircled{\tiny\#}}},$$
where $b^{\pi}=1-b^{\tiny{\textcircled{\tiny\#}}}b$.
\end{thm}
\begin{proof} Suppose $a,b\in R^{\tiny{\textcircled{\tiny\#}}}$, then by Lemma \ref{five-equations}, we have $a,b\in R^\#$ with $a^\#=(a^{\tiny{\textcircled{\tiny\#}}})^{2}a$ and $b^\#=(b^{\tiny{\textcircled{\tiny\#}}})^{2}b$, respectively. Since $ab=0$, $a+b\in R^\#$ by Lemma \ref{additive-group} and
\begin{equation*}
\begin{split}
&~~~~(a+b)^\#\\
&=(1-bb^\#)a^\#+b^\#(1-aa^\#)\\
&=(1-b^\#b)a^\#+b^\#(1-a^\#a)\\
&=(1-(b^{\tiny{\textcircled{\tiny\#}}})^{2}b^{2})(a^{\tiny{\textcircled{\tiny\#}}})^{2}a+
(b^{\tiny{\textcircled{\tiny\#}}})^{2}b(1-(a^{\tiny{\textcircled{\tiny\#}}})^{2}a^{2}).\\
&=(1-b^{\tiny{\textcircled{\tiny\#}}}b)(a^{\tiny{\textcircled{\tiny\#}}})^{2}a+
(b^{\tiny{\textcircled{\tiny\#}}})^{2}b(1-a^{\tiny{\textcircled{\tiny\#}}}a).
\end{split}
\end{equation*}
Since $ab=0$ and $a^{\ast}b=0$, then
\begin{equation*}
\begin{split}
&ab^{\tiny{\textcircled{\tiny\#}}}=ab(b^{\tiny{\textcircled{\tiny\#}}})^{2}=0.\\
&b^{\tiny{\textcircled{\tiny\#}}}a
=b^{\tiny{\textcircled{\tiny\#}}}bb^{\tiny{\textcircled{\tiny\#}}}a
=b^{\tiny{\textcircled{\tiny\#}}}(bb^{\tiny{\textcircled{\tiny\#}}})^{\ast}a
=b^{\tiny{\textcircled{\tiny\#}}}(b^{\tiny{\textcircled{\tiny\#}}})^{\ast}b^{\ast}a
=b^{\tiny{\textcircled{\tiny\#}}}(b^{\tiny{\textcircled{\tiny\#}}})^{\ast}(a^{\ast}b)^{\ast}
=0.\\
&a^{\tiny{\textcircled{\tiny\#}}}b
=a^{\tiny{\textcircled{\tiny\#}}}aa^{\tiny{\textcircled{\tiny\#}}}b
=a^{\tiny{\textcircled{\tiny\#}}}(aa^{\tiny{\textcircled{\tiny\#}}})^{\ast}b
=a^{\tiny{\textcircled{\tiny\#}}}(a^{\tiny{\textcircled{\tiny\#}}})^{\ast}a^{\ast}b
=0.
\end{split}
\end{equation*}
Let $x=(1-b^{\tiny{\textcircled{\tiny\#}}}b)a^{\tiny{\textcircled{\tiny\#}}}+
b^{\tiny{\textcircled{\tiny\#}}}.$ We will show that $x$ is a $\{1,3\}$-inverse of $a+b$.
\begin{equation*}
\begin{split}
&~~~~(a+b)x\\
&=(a+b)[(1-b^{\tiny{\textcircled{\tiny\#}}}b)a^{\tiny{\textcircled{\tiny\#}}}+
b^{\tiny{\textcircled{\tiny\#}}}]\\
&=a(1-b^{\tiny{\textcircled{\tiny\#}}}b)a^{\tiny{\textcircled{\tiny\#}}}+b(1-b^{\tiny{\textcircled{\tiny\#}}}b)a^{\tiny{\textcircled{\tiny\#}}}+
ab^{\tiny{\textcircled{\tiny\#}}}+bb^{\tiny{\textcircled{\tiny\#}}}\\
&=a(1-b^{\tiny{\textcircled{\tiny\#}}}b)a^{\tiny{\textcircled{\tiny\#}}}+bb^{\tiny{\textcircled{\tiny\#}}}\\
&=aa^{\tiny{\textcircled{\tiny\#}}}+bb^{\tiny{\textcircled{\tiny\#}}}.\\
&~~~~(a+b)x(a+b)\\
&=(aa^{\tiny{\textcircled{\tiny\#}}}+bb^{\tiny{\textcircled{\tiny\#}}})(a+b)
= aa^{\tiny{\textcircled{\tiny\#}}}a+aa^{\tiny{\textcircled{\tiny\#}}}b+bb^{\tiny{\textcircled{\tiny\#}}}a+bb^{\tiny{\textcircled{\tiny\#}}}b\\
&=aa^{\tiny{\textcircled{\tiny\#}}}a+bb^{\tiny{\textcircled{\tiny\#}}}b=a+b.
\end{split}
\end{equation*}
Hence by Theorem \ref{coreinverse2}, we have
\begin{equation*}
\begin{split}
&~~~~(a+b)^{\tiny{\textcircled{\tiny\#}}}\\
&=(a+b)^\#(a+b)(a+b)^{(1,3)}\\
&=[(1-b^{\tiny{\textcircled{\tiny\#}}}b)(a^{\tiny{\textcircled{\tiny\#}}})^{2}a+
(b^{\tiny{\textcircled{\tiny\#}}})^{2}b(1-a^{\tiny{\textcircled{\tiny\#}}}a)](a+b)[(1-b^{\tiny{\textcircled{\tiny\#}}}b)a^{\tiny{\textcircled{\tiny\#}}}+
b^{\tiny{\textcircled{\tiny\#}}}]\\
&=[(1-b^{\tiny{\textcircled{\tiny\#}}}b)(a^{\tiny{\textcircled{\tiny\#}}})^{2}a^{2}
+(1-b^{\tiny{\textcircled{\tiny\#}}}b)(a^{\tiny{\textcircled{\tiny\#}}})^{2}ab+
(b^{\tiny{\textcircled{\tiny\#}}})^{2}b(1-a^{\tiny{\textcircled{\tiny\#}}}a)a+
(b^{\tiny{\textcircled{\tiny\#}}})^{2}b\\
&~~~~(1-a^{\tiny{\textcircled{\tiny\#}}}a)b]
[(1-b^{\tiny{\textcircled{\tiny\#}}}b)a^{\tiny{\textcircled{\tiny\#}}}+
b^{\tiny{\textcircled{\tiny\#}}}]\\
&=[(1-b^{\tiny{\textcircled{\tiny\#}}}b)a^{\tiny{\textcircled{\tiny\#}}}a
+b^{\tiny{\textcircled{\tiny\#}}}b][(1-b^{\tiny{\textcircled{\tiny\#}}}b)a^{\tiny{\textcircled{\tiny\#}}}+
b^{\tiny{\textcircled{\tiny\#}}}]\\
&=(1-b^{\tiny{\textcircled{\tiny\#}}}b)a^{\tiny{\textcircled{\tiny\#}}}a(1-b^{\tiny{\textcircled{\tiny\#}}}b)a^{\tiny{\textcircled{\tiny\#}}}
+b^{\tiny{\textcircled{\tiny\#}}}b(1-b^{\tiny{\textcircled{\tiny\#}}}b)a^{\tiny{\textcircled{\tiny\#}}}+
(1-b^{\tiny{\textcircled{\tiny\#}}}b)a^{\tiny{\textcircled{\tiny\#}}}ab^{\tiny{\textcircled{\tiny\#}}}+
b^{\tiny{\textcircled{\tiny\#}}}bb^{\tiny{\textcircled{\tiny\#}}}\\
&=(1-b^{\tiny{\textcircled{\tiny\#}}}b)a^{\tiny{\textcircled{\tiny\#}}}+b^{\tiny{\textcircled{\tiny\#}}}\\
&=b^{\pi}a^{\tiny{\textcircled{\tiny\#}}}+b^{\tiny{\textcircled{\tiny\#}}},
\end{split}
\end{equation*}
where $b^{\pi}=1-b^{\tiny{\textcircled{\tiny\#}}}b$.
\end{proof}

\begin{eg} \label{example2}
The condition $ab=0$ in Theorem $\ref{core-additive1}$ cannot be dropped. Let $R=M_{2}(\mathcal{F})$ with transpose as involution,
where $\mathcal{F}$ is a field. Take
 $a=\left[\begin{smallmatrix}
              1  & 0 \\
              -1 & 0
       \end{smallmatrix}
       \right]$\normalsize
 and
 $b=\left[\begin{smallmatrix}
              -1 & 0 \\
              -1 & 0
       \end{smallmatrix}
       \right]$\normalsize
which satisfy $a^{2}=a$ and $b^{2}=-b$, then $a,b$ is group invertible.
It is easy to see that $a,b\in R^{\{1,3\}}$, thus by Theorem \ref{coreinverse2} $a,b$ is core invertible. Yet,
 $a+b=\left[\begin{smallmatrix}
              0  & 0 \\
              -2 & 0
       \end{smallmatrix}
       \right]$\normalsize
 which satisfies
  $(a+b)^{2}=\left[\begin{smallmatrix}
              0  & 0 \\
              0  & 0
       \end{smallmatrix}
       \right]$\normalsize would imply that $a+b$ is not group invertible and
therefore $a+b$ is not core invertible by  Theorem $\ref{coreinverse2}$.
\end{eg}

If $a,b\in R^{\dagger}$ with $a^{\ast}b=0$ and $ab^{\ast}=0$, then $a+b\in R^{\dagger}$ and $(a+b)^{\dagger}=a^{\dagger}+b^{\dagger}.$
This result was proved by Penrose \cite[Lemma 1]{P} in the complex matrix case, but it is valid for the ring case.
Yet it is not true for core inverse. The counterexample can be found as follows.

\begin{rem} \label{exampleabba}
Let $R=M_{2}(\mathbb{Z}_{4})$ with transpose as involution.
 $a=\left[\begin{smallmatrix}
              -1  & 1 \\
              0   & 0
       \end{smallmatrix}
       \right]$\normalsize
 and
 $b=\left[\begin{smallmatrix}
              0 & 0 \\
              1 & 1
       \end{smallmatrix}
       \right]$\normalsize
$~$in $R$,
which satisfy $a^{2}=-a$ and $b^{2}=b$, then $a,b$ is group invertible.
It is easy to see that $a\in Ra^{\ast}a$ and $b\in Rb^{\ast}b$, thus
$a,b\in R^{\{1,3\}}$, which imply $a,b\in R^{\tiny\textcircled{\tiny\#}}$.
Yet, $a+b\notin R(a+b)^{\ast}(a+b)$, that is $a+b\notin R^{\{1,3\}}$,
thus $a+b\notin R^{\tiny\textcircled{\tiny\#}}$, although we have
$a^{\ast}b=0$ and $ab^{\ast}=0.$  This remark also shows that
the condition $ab=0$ in Theorem $\ref{core-additive1}$ cannot be dropped.
\end{rem}

\begin{rem}
Let $R=M_{2}(\mathcal{F})$ with transpose as involution,
where $\mathcal{F}$ is a field. Take
$a=\left[\begin{smallmatrix}
              1  & 0 \\
              0  & 0
       \end{smallmatrix}
       \right] and $~$
b=\left[\begin{smallmatrix}
              0   & 0 \\
              -1  & 0
       \end{smallmatrix}
       \right]$\normalsize,
then $a,b\in R^{\dagger}$ and $a^{\ast}b=ab=0$, yet
$a+b\notin R(a+b)^{\ast}(a+b)$, that is $a+b\notin R^{\{1,3\}}$,
thus $a+b\notin R^{\tiny{\textcircled{\tiny\#}}}$.
\end{rem}

\begin{cor} \label{core-additive2}
Let $a,b\in R^{\tiny{\textcircled{\tiny\#}}}$ with core inverses $a^{\tiny{\textcircled{\tiny\#}}}$ and $b^{\tiny{\textcircled{\tiny\#}}}$, respectively. If $ab=0=ba$ and $a^{\ast}b=0$, then $a+b\in R^{\tiny{\textcircled{\tiny\#}}}$.\\
Moreover,
$$(a+b)^{\tiny{\textcircled{\tiny\#}}}=a^{\tiny{\textcircled{\tiny\#}}}+b^{\tiny{\textcircled{\tiny\#}}}.$$
\end{cor}

Similarly, we have the following results for dual core inverse.

\begin{thm} \label{dual-core-additive1}
Let $a,b\in R_{\tiny{\textcircled{\tiny\#}}}$ with dual core inverses $a_{\tiny{\textcircled{\tiny\#}}}$ and $b_{\tiny{\textcircled{\tiny\#}}}$, respectively. If $ab=0$ and $ab^{\ast}=0$, then $a+b\in R_{\tiny{\textcircled{\tiny\#}}}$.\\
Moreover,
$$(a+b)_{\tiny{\textcircled{\tiny\#}}}=a_{\tiny{\textcircled{\tiny\#}}}+b_{\tiny{\textcircled{\tiny\#}}}a^{\pi},$$
where $a^{\pi}=1-aa_{\tiny{\textcircled{\tiny\#}}}$.
\end{thm}

\begin{cor} \label{dual-core-additive2}
Let $a,b\in R_{\tiny{\textcircled{\tiny\#}}}$ with dual core inverses $a_{\tiny{\textcircled{\tiny\#}}}$ and $b_{\tiny{\textcircled{\tiny\#}}}$, respectively. If $ab=0=ba$ and $ab^{\ast}=0$, then $a+b\in R_{\tiny{\textcircled{\tiny\#}}}$.\\
Moreover,
$$(a+b)_{\tiny{\textcircled{\tiny\#}}}=a_{\tiny{\textcircled{\tiny\#}}}+b_{\tiny{\textcircled{\tiny\#}}}.$$
\end{cor}

\vspace{0.2cm} \noindent {\large\bf Acknowledgements}

This research is supported by the National Natural Science Foundation of China (No.11201063 and No.11371089), the Specialized Research Fund for the Doctoral Program of Higher Education (No.20120092110020); the Jiangsu Planned Projects for Postdoctoral Research Funds (No.1501048B); the Natural Science Foundation of Jiangsu Province (No.BK20141327).

\end{document}